\documentclass[12pt]{amsart}
\usepackage[colorlinks=false,backref]{hyperref}

\usepackage{amsmath, amsbsy, enumerate,amssymb, amsfonts, latexsym, mathrsfs}
\usepackage{amsthm}
\usepackage{stmaryrd}
\usepackage{bm}
\usepackage[margin=26mm]{geometry}
\usepackage{xcolor}
\usepackage[all]{xy}

\newtheorem{theorem}{Theorem}[section]
\newtheorem{thm}{Theorem}[section]
\newtheorem{corollary}[theorem]{Corollary}
\newtheorem{cor}[thm]{Corollary}
\newtheorem{lemma}[thm]{Lemma}
\newtheorem{lem}[thm]{Lemma}

\newtheorem{proposition}[thm]{Proposition}

\theoremstyle{definition}

\newtheorem{rem}[theorem]{Remark}

\makeatletter
\newcommand{\imod}[1]{\allowbreak\mkern4mu({\operator@font mod}\,\,#1)}
\makeatother

\numberwithin{equation}{section}

\DeclareMathOperator{\GL}{GL}

\DeclareMathOperator{\Sp}{Sp}

\renewcommand{\bf}{\textbf}

 \renewcommand{\to}{\rightarrow}

\newcommand{\ol}{\overline}
\newcommand{\Spec}{\mathrm{Spec}}
\newcommand{\gm}{\mathfrak{m}}
\newcommand{\gB}{\mathfrak{B}}
\newcommand{\gC}{\mathfrak{C}}
\newcommand{\gG}{\mathfrak{G}}
\newcommand{\gH}{\mathfrak{H}}

\newcommand{\gU}{\mathfrak{U}}
\newcommand{\gX}{\mathfrak{X}}
\newcommand{\GG}{\mathbb{G}}
\newcommand{\FF}{\mathbb{F}}
\newcommand{\cB}{\mathcal{B}}

\begin{document}
\title[]{Unipotent Subgroups of Stabilizers}

\author[P. Gille]{Philippe Gille}\address{UMR 5208
Institut Camille Jordan - Universit\'e Claude Bernard Lyon 1
43 boulevard du 11 novembre 1918
69622 Villeurbanne cedex - France 
}

\email{gille@math.univ-lyon1.fr}

\author[R. M.  Guralnick]{Robert Guralnick}
\address{Department of Mathematics, University of Southern California, Los Angeles, CA 90089-2532, USA}
\email{guralnic@usc.edu}


\date{\today}

\dedicatory{For James Humphreys} 

\thanks{The second author was partially supported  by the NSF grant DMS-1901595 and a Simons Foundation Fellowship 609771.}

\begin{abstract}    We consider semi-continuity of certain dimensions on group schemes. 
\end{abstract}

\maketitle

\section{Introduction}    

Let $G$ be an algebraic group over a field $k$. 
Let $d_u(G)$ the maximal dimension of a (smooth) connected unipotent subgroup of $G_{\ol{k}}$.  
Using techniques \`a la Demazure-Grothendieck, we 
show the following result.

\begin{thm} \label{t:unipotent rank}
Let $S$ be a scheme and let $G$ be a separated $S$-group scheme  of finite presentation (for example an affine 
$S$-group of finite presentation). Then the function
$d_u$ on $S$ is upper semi-continuous. 
\end{thm}

Upper semi-continuous means informally that the function 
jumps along closed sets.
In particular, the function is 
locally constant at the points of the minimal value locus. This gives the following useful corollary. 

\begin{cor} \label{c:unipotent}  Let $S$ be an irreducible  scheme and let $G$ be a separated  $S$-group scheme   of  finite presentation.
If for the generic point $\xi \in S$,  $G_{\ol{\kappa(\xi)}}$ contains a $d$-dimensional smooth unipotent subgroup, then the same is true
for  $G_{\ol{\kappa(s)}}$
for all $s \in S$.  
\end{cor}

A case of special interest is the    group  scheme of stabilizers called
also the stabilizer of the diagonal \cite[V.10.2]{SGA3}. More precisely if $G$ is an $S$--group
scheme acting on a separated $S$--scheme $X$ of finite presentation
(with $S$ noetherian), we consider the fiber product
\[
\xymatrix{
 F \ar[r] \ar[d]&   X \ar[d]^\Delta \\
G \times_S X \ar[r] & X \times_S X & (g,x) \ar@{|->}[r] & (x,g.x).
}\]
It defines an $X$-group scheme $F$ which is a closed $X$-subgroup scheme of $G \times_S X$
of finite presentation
such that for each $x \in X$ of image $s \in S$, 
$F_{\kappa(x)} \subset G \times_{\kappa(s)}{\kappa(x)}$ is the stabilizer of the point $x$
for the action of $G \times_{\kappa(s)}{\kappa(x)}$ on $X \times_{\kappa(s)}{\kappa(x)}$.
In case $\kappa(s)=\kappa(x)$, the usual notation for $F_{\kappa(x)}$ is $G_x$.

One motivation for this question was
related to the base size  of finite groups acting primitively on a set and the existence of regular orbits for nontransitive actions.
An important  case is that of a finite simple group of Lie type over a finite field where the action comes from the algebraic group.
See \cite{BGS}. 
Another interesting case is when $G$ is a reductive algebraic group acting linearly on $X$ (or acting on the Grassmanian of a 
rational module).   See \cite{GG, GL,  PV} for more on this.
We give more details on this in Section \ref{s:base}. 

     Note that by the Lang-Steinberg theorem,  an algebraic group defined over $\FF_q$ has
 a Borel subgroup defined over $\FF_q$.   We also know that 
 if $U$ is a $d$-dimensional unipotent subgroup defined over $\FF_q$,  then $|U(\FF_q)| =q^d$ \cite{B,GG}.
 We   can apply Corollary \ref{c:unipotent} to the stabilizer scheme to obtain the following result. 
 
\begin{cor} \label{c:finite}   Let $G$ be a  algebraic group
 acting faithfully on an irreducible variety $X$ 
and assume that the $G,X$ and the action are defined over a finite field $\mathbb{F}_q$.   
Assume that there is a nonempty open subset $X_0$ of $X$ such that the stabilizer $G_x$ of $x \in X_0$ has a  $d$-dimensional
unipotent subgroup. 
%%% which are commutative of the right dimension (no need to assume commutative).
Then for all $x \in X(\FF_q)$,  $G_x(\FF_q)$ contains a subgroup of size $q^d$. 
\end{cor}

One can ask more generally what other functions are upper semi-continuous.  Of
course, dimension is \cite[VI$_B$.4.3]{SGA3}.  In fact, we will show that other such functions are also upper semi-continuous.
 On the other hand,  if we define $d^0(G)$ to be the dimension of the derived
subgroup of the connected component of $G$, it is not true that $d^0$ need be upper semi-continuous.  
We study this in Section \ref{s:derived} and particularly for the smooth case. 
Smoothness rarely holds in the case of the stabilizer scheme and we give an example to show
the failure of upper semi-continuity for $d^0(G)$ in stabilizer schemes.

\smallskip

We use mostly the terminology  of 
Borel's book  \cite{B} and time to time
the more general setting
Demazure-Gabriel's book \cite{DG} and 
the SGA3 seminar \cite{SGA3} where in particular an algebraic group  is not supposed smooth.
All definitions are coherent.

In the next sections, we prove some preliminary results.  We prove Theorem \ref{t:unipotent rank} 
and other upper semi-continuity results in Section \ref{s:upper}.   In Section
\ref{s:derived}, we consider the dimension of the derived subgroup.  In the final section, we present
an alternate proof of Theorem \ref{t:unipotent rank} due to Brian Conrad. 
 
 \medskip

 \noindent{\bf{Acknowledgments}}: We thank 
 Jean-Pierre Serre,  Brian Conrad and Matthieu Romagny for valuable comments on a preliminary version  of the paper.
 We also thank Brian Conrad for allowing  us to include his alternate proof of Theorem \ref{t:unipotent rank}
 and Skip Garibaldi for the applications to the abelianization (Cor. \ref{cor_ab}).

\section{Definition of the rank functions}

Let $k$ be a field and 
let $G/k$ be an algebraic group.
We remind the reader that 
 a linear algebraic group $G$ is
unipotent if for each (or any) faithful
$\ol{k}$--representation
$\rho: G_{\ol k} \to \GL_{n,\ol{k}}$, 
$\rho( G(\ol k))$  consists in unipotent elements.
This agrees with the general definition given 
in \cite[4.2.2.1]{DG}, see 4.3.26.b of this reference.
In practice we use the equivalent definition
that $G$ admits a closed $k$-embedding in 
a $k$--group of strictly upper triangular matrices 
({\it ibid}, 4.2.2.5).

Similarly a  closed smooth $k$--subgroup $G$ of $\GL_n$
is trigonalizable if there exists $h \in \GL_n(k)$
such that $hGh^{-1} \subset B_n$ where $B_n$
is the $k$--Borel subgroup of $\GL_n$ 
consisting in upper triangular matrices; the same
holds for any linear representation $G \to \GL_r$
\cite[15.5]{B}.
That definitions holds more for an arbitrary affine
algebraic $k$--group since it is equivalent to the
Demazure-Gabriel's definition \cite[4.2.3.4]{DG}.

\subsection{Relative version}
 We define the following invariants:

\begin{enumerate}
\item  $D_u(k,G)= \hbox{Maximal dimension of an unipotent $k$-subgroup of $G$}$; 

\smallskip

\item $ D_{cu}(k,G)= \hbox{Maximal dimension of a  commutative unipotent $k$-subgroup of $G$}$;  

\smallskip 
 
\item
$ D_{t}(k,G)= \hbox{Maximal dimension of a
 trigonalizable $k$-subgroup of $G$}$;

\smallskip

\item $ D_n(k,G)= \hbox{Maximal dimension of an affine nilpotent $k$-subgroup of $G$}$;

\smallskip

\item $ D_{nt}(k,G)= \hbox{Maximal dimension of a nilpotent trigonalizable $k$-subgroup of $G$}$; 

\smallskip 

\item
$D_s(k,G)= \hbox{Maximal dimension of an affine solvable $k$-subgroup of $G$} $;

\smallskip

\item $ D_c(k,G)= \hbox{Maximal dimension of an affine  commutative $k$-subgroup of $G$}$;

\smallskip

\item $D'_c(k,G)= \hbox{Maximal dimension of a commutative $k$-subgroup of $G$}$;  

\smallskip

\item $ D'_n(k,G)= \hbox{Maximal dimension of a nilpotent $k$-subgroup of $G$} $;

\smallskip

\item $ D'_s(k,G)= \hbox{Maximal dimension of a solvable $k$-subgroup of $G$}$.  
\end{enumerate}

We have $D_c(k,G) \leq D'_c(k,G)$, $D_n(k,G) \leq D'_n(k,G)$
with equalities if $G$ is affine.  
We have the obvious inequalities $D'_c(k,G) \leq D'_n(k,G) \leq D'_s(k,G)$
and $D_{cu}(k,G) \leq \mathrm{Min}(D_{c}(k,G), D_{u}(k,G))$
and $D_u(k,G)  \leq  D_n(k,G)$,
 $D_{nt}(k,G) \leq \mathrm{Max} \bigl(D_n(G), D_t(k,G) \bigr) \leq D_s(k,G)$.

All these functions are increasing by change of fields.

\begin{lem} \label{lem_alg_closed}
 If $k$ is algebraically closed, then $D(k,G)=D(F,G)$ for 
 any field extension $F/k$ and for each function $D$ as above.
 \end{lem}
 
 \begin{proof} We do it for $d_u$, the other cases being similar. 
We can assume that $F$ is algebraically closed
 so that $G_F$ admits a smooth unipotent $F$--subgroup $U$ of dimension $d$.
There exists a finitely generated $k$--subextension $E$ of $F$ 
such that $U$ is defined over $E$.
The field $E$ is 
function field of a smooth $k$-variety $X$. Up to shrinking $X$,
$U$ extends to a closed subgroup scheme $\gU$ of $G \times_k X$
which is smooth in view of \cite[VI$_B$.10]{SGA3}.
Up to shrinking one more time, $\gU$ is a closed subgroup scheme of the $X$-group scheme of 
strictly upper triangular matrices. Since $k$ is algebraically closed,
we have $X(k) \not = \emptyset$. The fiber at $x \in X(k)$
provides a smooth unipotent $k$--subgroup  $\gU_x \subset G$
of dimension $d$.
Thus  $D(k,G) \geq d$.
 \end{proof}

\subsection{Absolute version}

We define now 

\begin{enumerate}
\item  $d_u(G)= D_u(\ol{k},G)$, i.e.  the maximal  dimension of an unipotent $\ol k$-subgroup of $G_{\ol k}$; 

\smallskip

\item $d_{cu}(G)= D_{cu}(\ol{k},G)$; 

\smallskip 
 
\item $d_{t}(G)= D_{t}(\ol{k},G)$;

\smallskip

\item $d_{n}(G)= D_{n}(\ol{k},G)$;

\smallskip

\item $d_{nt}(G)= D_{nt}(\ol{k},G)$; 
\smallskip 

\item $d_{s}(G)= D_{s}(\ol{k},G)$;

\smallskip

\item $d_{c}(G)= D_{c}(\ol{k},G)$; 

\smallskip

\item $d'_{c}(G)= D'_{c}(\ol{k},G)$; 

\smallskip

\item $d'_{n}(G)= D'_{n}(\ol{k},G)$;

\smallskip

\item  $d'_{s}(G)= D'_{s}(\ol{k},G)$.  
\end{enumerate}

Clearly it does not depend of the choice of $\ol{k}$.
All these functions  are insensitive to  change of fields.

\begin{lem} \label{lem_alg_closed2}
 We have  $d(k,G)=d(F,G)$ for 
 any field extension $F/k$ and for each function $d$ as above.
 \end{lem}

\begin{proof}
The function $d$ is the absolute version of a
relative rank function $D$.
Let $\ol{F}$ be an algebraic closure of $F$ containing 
$\ol{k}$. Lemma \ref{lem_alg_closed} shows that 
$D(\ol{k},G)=D(\ol{F},G)$ whence  $d(k,G)=d(F,G)$.
\end{proof}

Up to considering the connected reduced fiber, we have also 

\vskip-4mm

$$ d_u(G)= \hbox{Maximal dimension of a smooth connected unipotent  $\ol{k}$-subgroup of $G_{\ol{k}}$} ;
$$  
  and similarly for the other absolute rank functions.
Since affine smooth connected solvable  $\ol{k}$--subgroups
 trigonalizable by the Lie-Kolchin's theorem \cite[10.5]{B}
it follows that $d_s(k,G)=d_t(k,G)$ and similarly we have  $d_n(k,G)=d_{nt}(k,G)$.

We have $d_c(G) \leq d'_c(k,G)$, $d_n(G) \leq d'_n(G)$ and $d_c(G) \leq d'_c(G)$ 
with equalities if $G$ is affine.  
We have the obvious inequalities $d'_c(G) \leq d'_n(G) \leq d'_s(G)$
and $d_{cu}(G) \leq \mathrm{Min}(d_{c}(G), d_{u}(G))$
and $d_u(G)  \leq  d_n(G)  \leq d_s(G)$.

\subsection{Connections with the literature}
  
\noindent (a) In \cite[XII.1]{SGA3}, Grothendieck defines 
related  rank functions but which are different. For example 
the Grothendieck unipotent rank $\rho_u(G)$ of
a smooth connected group $G$ over  an algebraically closed
field is $d_u(C)$ where $C$ is a Cartan subgroup of $G^0$.
This function $\rho_u$ is upper semi-continuous
if $G$ is smooth affine over a base \cite[XII.2.7.(i)]{SGA3}.
Our result does not require smoothness (nor flatness).
 
\smallskip

\noindent (b) If $G$ is a smooth connected affine algebraic group over an algebraically closed field $k$,
the Borel subgroups are the maximal
smooth connected solvable subgroups, they are all conjugate
and so  $d_s(G)$ is nothing but the dimension of a Borel subgroup. The unipotent radicals are  
then  the maximal
smooth connected unipotent subgroups,
they  are all conjugate and so $d_u(G)$ is the dimension of
a the unipotent radical of a Borel subgroup of $G$.
For $d_n(G)$ the situation is more complicated
since maximal smooth connected nilpotent subgroups of $G$
do not consist of a single   conjugacy class. However
there are finitely many conjugacy classes (Platonov, 
\cite[thm 2.13]{P}).

\section{Specialization over a regular local  ring}

\subsection{Group schemes over a  DVR} \label{subsec_DVR}
Let $A$ be a discrete valuation ring of fraction field $K$
and of residue field $k$. If $G$ is an $A$--group scheme  of finite 
presentation, we would like
to list properties on the generic fiber $G_K$  which are inherited by the closed fiber $G_k$.
For example, if $G$ is flat, then $G_K$ and $G_k$ share the same dimension \cite[VI$_B$.4.3]{SGA3}.
Also if $G$ is separated and flat and $G_K$ is affine, then $G$ is affine (Raynaud, \cite[prop 3.1]{PY}) so that $G_k$ is affine.
Flatness is then an important property, we recall that an $A$--scheme $\gX$ is flat if and only if 
it is torsion free, this second condition being equivalent to the density of the generic
fiber $\gX_K$ in $\gX$ \cite[\S 14.3]{GW}. For the study of the function $d_u$,
the next statement is the key step.

\begin{lemma}  \label{lem_DVR_uni} We assume that $G$ is flat
and affine.
If $G_K$ is trigonalizable (resp.\, unipotent) so is $G_k$.
\end{lemma}

\begin{proof} We assume  that $G_K$ is trigonalizable, that is, $G_K$ is an extension
 of a diagonalizable $K$-group by an unipotent $K$-group. 
 According to \cite[1.4.5]{BT}, there exists a closed monomorphism $\rho: G \to \GL_N$.
 Since $G_K$ is trigonalizable, its stabilizes a flag of $K^N$ \cite[prop. 4.2.3.4, $(i) \Longrightarrow (iv)$]{DG}.
 In other words $\rho$ factorizes through a Borel subgroup $B_K$ of $\GL_N$.
 Since the $A$--scheme of Borel subgroups of $\GL_N$ is projective \cite[XXII.5.8.3]{SGA3}, $B_K$ extends uniquely
 to a Borel  $A$--subgroup scheme $B$ of $\GL_N$
 [a concrete way is to use a filtration $V_0=0 \subsetneq V_1 \subsetneq V_2 \dots \subsetneq V_N=K^N$ and
 put $\widetilde V_i=A^n \cap V_i$ for each $i$. It defines an $A$-flag of lattices
 of $A^n$ which is stabilized by  $G$]. Its reduction to $k$ provides 
 an embedding of $G_k$ to a Borel  $k$--subgroup of $\GL_{N,k}$.
 The last quoted result,  $(v) \Longrightarrow (i)$,  enables us to conclude that $G_k$  is trigonalizable.

We assume furthermore that  $G_K$ is unipotent.  We have an exact sequence of $A$--group schemes
$1 \to U \to B \xrightarrow{q} \GG_m^N \to 1$.
Since $G_K$ is unipotent $q_K \circ \rho_K= G_K \to (\GG_m^N)_K$ is trivial  according
to \cite[IV.2.2.4]{DG}. Since $G_K$ is dense in $G$  it follows that 
$q\circ \rho= G \to \GG_m^N$ is trivial so that $\rho$ factorizes through the unipotent radical
$U$ of $B$. A fortiori $G_k$ admits a representation in a strictly upper triangular
$k$-group so is unipotent according to \cite[IV.2.2.5, $(vi) \Longrightarrow (i)$]{DG}. 
\end{proof}

\begin{rem}{\rm 
(a) If $G_K$ is split $K$--unipotent, a result of
 Ve\v{\i}sfe\v{\i}ler-Dolgachev on unipotent group
schemes \cite[thm. 1.1]{VD} shows also that $G_k$ is unipotent. This is a very different proof.

\smallskip

\noindent (b) If $G$ has smooth fibers and $G_K$ is unipotent, the fact that $G_k$
 is unipotent follows from \cite[VI$_B$.8.4.(ii)]{SGA3}. This is a quite  different proof.

\smallskip

\noindent (c) One simpler proof of (b) occurs in 
the alternate  proof below of Theorem \ref{t:unipotent rank} using the smoothness 
of the scheme of maximal tori of $G$, see \S \ref{sec_brian}. 
 
 }
\end{rem}

For dealing the other rank functions, we need more facts.

\begin{lemma}  \label{lem_DVR0}
 Let $H$ be a $K$--subgroup of $G_K$ and let $\gH$ be the schematic 
 closure of $H$ in $G$.
 
 \smallskip
 
 \noindent (1) $\gH$ is a closed $A$--subgroup scheme of $G$ which is flat 
 of finite presentation. If 
 $H$ is central, then $\gH$ is  central.

 \smallskip
 
 \noindent (2) The fppf quotient $G/\gH$ is representable by a separated $A$-scheme of finite presentation.
 Furthermore if $G$ is flat so is $G/\gH$.

 \smallskip
 
 \noindent (3) If $H$ is normal in $G_K$, then $\gH$ is a normal $A$--subgroup 
 scheme of $G$ and $G/\gH$ carries a natural structure of $A$-group scheme.

\end{lemma}

\begin{proof}
 (1) The first part is in \cite[(2.8.1)]{EGA4}.
 Assume that $H$ is central, that is the commutator map 
   ${G_K \times_K H} \to G_K$ is trivial.
 Since $G_K \times_K H$ is dense in $G \times_A \gH$, it follows that 
 the commutator map  $G \times_A \gH \to G$ is trivial, so that $\gH$ is central in $G$.
 
 \smallskip
 
 \noindent (2) The representability is result by Anantharaman \cite[IV, th. 4.C]{A}
 so that  $G/\gH$ is separated and  of finite presentation  \cite[VI$_B$.9.2.(x) and (xiii)]{SGA3}.
If $G$ is flat,   $G/\gH$ is flat according to  \cite[VI$_B$.9.2.(xi)]{SGA3}.
 
 \smallskip
 
 \noindent (3) Assume that $H$ is normal in $G_K$, that is, the commutator map \break $G_K \times_K H \to G_K/H$ is trivial.
 Since $G_K \times_K H$ is dense in $G \times_A \gH$, it follows that  
 the commutator map $G \times_A \gH \to G/\gH$ is trivial so that $\gH$ is a normal $A$--subgroup scheme of $\gG$.
 According to  \cite[VI$_B$.9.2.(iv)]{SGA3},  it follows that $G/\gH$ carries a natural structure of $A$-group scheme.
 
\end{proof}

\begin{lemma}  \label{lem_DVR} We assume that $G$ is flat.
 If $G_K$ is 
commutative (resp. nilpotent, solvable).  
Then $G_k$  is  commutative (resp. nilpotent, solvable).

\end{lemma}

\begin{proof} If  $G_K$ is commutative,  so is $G$ and $G_k$ according to Lemma \ref{lem_DVR0}.(1).

We assume now that $G_K$ is nilpotent, that is, admits a central composition  serie \break  $H_0=1 \subset H_1  \subset H_2 \subset \dots
 \subset H_{n-1} \subset H_n=G_K$ where the $H_i$'s are normal $K$--subgroups of $G_K$
 and such that each $H_{i+1}/H_i$ is central in $G_K/H_i$.
 Let $G_i$ be the schematic closure of $H_i$ in $G$, this is a flat $A$-group scheme and 
 all $H_i$'s are normal $A$--subgroups of $G$ according to Lemma \ref{lem_DVR0}.(3).
 Furthermore each quotient  $G_{i+1}/G_i$ is central in $G/H_i$.
 By extending the scalars to $k$ we get then a central composition series for $G_k$.
 
 The argument is similar for the solvable case.
\end{proof}

\subsection{The regular local ring case}

Let $A$ be a regular local ring with fraction field $K$ and residue field $k$. 

\begin{proposition} \label{prop_regular} Let $G$ be an  $A$-group scheme of finite 
presentation.
 
\smallskip

\noindent (1) Assume that $k$ is infinite and that $G_K$ contains an algebraic subgroup  (resp.\ normal subgroup) of dimension $d$.  
Then $G_k$ contains an algebraic subgroup (resp.\ normal subgroup) of dimension $d$.

\smallskip

\noindent (2) Assume that $G_K$ contains an algebraic subgroup which is 
commutative (resp. nilpotent, solvable) of dimension $d$.  
Then $G_k$ contains an algebraic subgroup which is 
commutative (resp. nilpotent, solvable) of dimension $d$.

\smallskip

\noindent (3)  Assume that that $G$ is separated and that $G_K$ contains
an algebraic subgroup which is affine commutative 
(resp.\  unipotent, commutative unipotent, affine nilpotent, nilpotent trigonalizable, trigonalizable,
affine solvable) of dimension $d$.  
Then $G_k$ contains an algebraic subgroup which is 
affine commutative (resp.\  unipotent, commutative unipotent, affine nilpotent, nilpotent trigonalizable, trigonalizable,
affine solvable) of dimension $d$.

\end{proposition}

\begin{proof} According to \cite[lemma 15.1.1.6]{EGA4}, there exists
a discrete valuation ring $B$ which dominates 
$A$ and such that its residue field  is a purely (finitely generated)
transcendental extension of $k$. We denote by $L$ the  fraction field of $B$
and by $l$ its residue field. We have $l=k$ or $k(t_1, \dots, t_n)$.

\smallskip

\noindent (1) Our assumption is that $G_K$ contains
an algebraic subgroup (resp.\ normal subgroup) $H$ which is  of dimension $d$.  We consider the schematic closure of $H_L$ in $G_B$, this
defines a flat $B$--group scheme (resp.\ normal $B$--subgroup scheme according to Lemma \ref{lem_DVR0}.(3)) 
$\gH$ of closed fiber $\gH_l$ which is a subgroup of $(G_k)_l$.
Since $k$ is infinite, one may ``specialize'' at a rational $k$--point to obtain
a $k$--subgroup of $G_k$ of dimension $d$.

\smallskip

\noindent (2) Our assumption is that $G_K$ contains
an algebraic subgroup $H$ which is  commutative (resp.\ nilpotent,   solvable) of dimension $d$.  We consider the schematic closure of $H_L$ in $G_B$, this
defines a flat $B$--group scheme $\gH$ of closed fiber $\gH_l$ which is a subgroup of $(G_k)_l$.
We apply Lemma \ref{lem_DVR} to $\gH$ and obtain that 
$\gH_l$ is  commutative (resp.\ nilpotent,   solvable) of dimension $d$.

By induction on $n$, we may assume that $l=k(t)$. We have that $G_{k((t))}$ admits 
the subgroup $\gH_{k((t))}$ which is   commutative (resp.\ nilpotent,   solvable) of dimension $d$.
By performing the same method as above in the case of the DVR $k[[t]]$, 
it follows that $G_k$ contains a commutative (resp.\ nilpotent,   solvable) subgroup  of dimension $d$.

\smallskip

\noindent (3) Though the  argument is very similar, we provide all details. Our assumption is that $G_K$ contains
an algebraic subgroup $H$ which is  affine commutative (resp.\  unipotent, commutative unipotent,  affine nilpotent, nilpotent trigonalizable, trigonalizable,
affine solvable).
 We consider the schematic closure of $H_L$ in $G_B$, this
defines a flat $B$--group scheme $\gH$ of closed fiber $\gH_l$ which is a subgroup of $(G_k)_l$.
Since $G_B$ is separated so is $\gH$. Raynaud's affiness criterion \cite[prop 3.1]{PY}
ensures that $\gH$ is affine over $B$.

We combine  Lemma \ref{lem_DVR_uni}  and \ref{lem_DVR}  for  $\gH$ and obtain that 
$\gH_l$ is affine commutative (resp.\   unipotent,  commutative  unipotent,  affine nilpotent, nilpotent trigonalizable, trigonalizable,
affine solvable) of dimension $d$.

Once again, by induction on $n$, we may assume that $l=k(t)$. We have that $G_{k((t))}$ admits 
the subgroup $M$ which is   affine commutative (resp.\  unipotent, affine nilpotent, nilpotent trigonalizable, trigonalizable, affine solvable).
By performing the same method as above in the case of the DVR $k[[t]]$, 
it follows that $G_k$ contains an affine (resp.\  unipotent, affine nilpotent, nilpotent trigonalizable, trigonalizable,
affine solvable) subgroup  of dimension $d$.
\end{proof}

\subsection{More permanence properties}

\begin{lem} \label{lem_permanence}
Let $G$ be an algebraic group defined over a field $k$.

\smallskip

\noindent (1) For each function $D$ as above, we have  $D(k,G)=  D(k(t),G)=  D(k((t)),G)$.
 
 \smallskip
 
 \noindent (2) If $X$ is a connected smooth $k$--variety such that $X(k) \not = \emptyset$, we have
 $D(k,G)= D(k(X),G)$.
 
\end{lem}

\begin{proof}
 (1) We have the obvious inequalities $D(k,G) \leq D(k(t),G) \leq  D(k((t)),G)$.
 The fact that $D(k((t)),G) \leq D(k,G)$ follows from Proposition 
 \ref{prop_regular} applied to $A=k[[t]]$ and the group scheme $G \times_k k[[t]]$.
 
 \smallskip
 
 \noindent (2) We have $D(k,G) \leq  D(k(X),G)$. For proving the converse
 inequality, we pick a point $x \in X(k)$. Let $(t_1,\dots, t_d)$  be a system of parameters 
 of the regular local ring $R=\mathcal{O}_{X,x}$.
 Then its completion is $k$--isomorphic to $k[[t_1,\dots, t_d]]$ which embeds in 
 the field of iterated Laurent series
 $k((t_1)) \dots ((t_d))$. It follows that $k(X)$ embeds in  $k((t_1)) \dots ((t_d))$
 so that  $D(k(X),G) \leq  D(k((t_1)) \dots ((t_d)),G)$. By induction on $d$, 
 (1) provides  $D(k,G)=D(k((t_1)) \dots ((t_d)),G)$ so that 
 $D(k(X),G) \leq D(k,G)$.

\end{proof}

\section{Upper semi-continuity}  \label{s:upper} 

\begin{thm} \label{thm_main}
% Let $S$ be a  scheme and let $G$ be an $S$-group scheme  of  finite presentation. For each function $d_\bullet$ as above we
define $d_\bullet(s) = d_\bullet(G_{\kappa(s)})$ for each $s \in S$.

\smallskip

\noindent (1) The functions $d'_c, d'_n, d'_s$ 
on $S$ arer upper semi-continuous. 

\smallskip

\noindent (2) Assume that $G$ is separated. The functions $d_c, d_u, d_{cu}, d_n,  d_s$ 
on $S$ arer upper semi-continuous. 

\end{thm}

\begin{proof}
 We prove both statements simultaneously.  Let $d_\bullet$ one of the function.
The problem is local so we can assume that $S=\Spec(A)$ 
for $A$ an integral local ring of fraction field $K$
and residue field $k$. We have to show that $d_\bullet(G_k) \geq d_\bullet(G_K)=d$.
By using the standard yoga of noetherian reduction \cite[VI$_B$.10.2]{SGA3}, 
we can assume that $A$ is furthermore noetherian.
If $A$ is a field, we have $K=k$ and this is obvious.
We assume that $A$ is not a field. 
According to \cite[prop. 7.1.7]{EGA2}, there exists an extension
$L$ of $K$ (of finite type) and equipped with a discrete valuation such that  its valuation ring $B$ dominates $A$, that is, 
$A \subset B$ and $\gm_B \cap A = \gm_A$. 
We denote by $l$ the residue field of $B$.
Since $d_\bullet(G_k)= d_\bullet(G_l)$ and $d_\bullet(G_K)=d_\bullet(G_L)$, 
it is enough to show that $d_\bullet(G_l) \geq d_\bullet(G_L)=d$.

In other words  the problem reduces to the case of  a discrete valuation ring
$A$ with  fraction field $K$.
Our assumption is that $G_{\ol{K}}$ contains a closed subgroup of dimension $d$
which is commutative (resp.\ nilpotent, solvable, affine commutative,
unipotent, commutative unipotent, nilpotent, nilpotent trigonalizable,
nilpotent, nilpotent trigonalizable, nilpotent, trigonalizable, solvable).
Then there exists a finite $K$--subextension $K'$ of $\ol{K}$
such that the same holds for $G_{K'}$.
Let $v'$ be an extension of the valuation $v_K$ to $K'$ and denote
by $B'$ its valuation ring and by $k'$ its residue field. 
Once again we have $d_\bullet(G_k)= d_\bullet(G_{k'})$ and $d=d_\bullet(G_K)=d_\bullet(G_{K'})$,
so that it is enough to show that $d_\bullet(G_{k'}) \leq d$.
Proposition \ref{prop_regular}.(2) and (3) shows that $G_{k'}$ contains a closed subgroup of dimension $d$
which is commutative (resp.\ nilpotent, solvable, affine commutative,
unipotent, commutative unipotent, nilpotent, nilpotent trigonalizable,
nilpotent, nilpotent trigonalizable, nilpotent, trigonalizable, solvable).
We get  then  $D_\bullet(k', G_k) \geq d$ so a fortiori
$d_\bullet(G_k) \geq d$.
\end{proof}

\section{Finite Groups}  \label{s:base}

One motivation for considering this question comes from a problem about finite groups.

Let $G$ be a group acting on a set $X$.   A base of $G$ acting on $X$ is a subset $Y$ of $X$ such 
any element of $g \in G$ which fixed $Y$ pointwise acts trivially
on $X$.    The base size $b(G,X)$ is the minimal cardinality of a base.  
In the case of finite groups, this has been  classical object of study for more than $150$ years.  
    This has had many applications (e.g. in computational group theory).  
One is also interested in this from a probabilistic point of view;  what is the proportion of the subsets of size $b$ which
are a base.  

In \cite{BGS},  this was considered for $G$ a simple algebraic group acting on the homogeneous space $G/M$ with $M$
a maximal closed subgroup and in almost cases $b(G,M)$ was determined exactly (in a few cases, there was a small range
of possible values).   In this case one can consider two other quantities.  We define $b^0(G,X)= c$ to be the smallest positive
integer $c$ so that there is subset  $Y$ of $X$ of size $c$ so that the pointwise stabilizer of $Y$ is finite and $b^1(G,X)=e$
where $e$ is the smallest positive integer such that the pointwise stabilizer of $e$ points is trivial.  It is easy to show
that $b^0 \le b \le b^1 \le b^0 + 1$.  

Note that this can be rephrased in terms of $G$ acting on $(G/M)^e$ and asking if there is a regular orbit or an orbit
of $\dim G$ or if the generic orbit is regular.  

More generally let $G$ be an algebraic group acting on a variety $X$ and assume that the action is defined over a finite field.  We are interested
the stabilizers in $G(\FF_q)$ of a point $ x \in X(\FF_q)$  (and more generally we consider Steinberg-Lang endomorphisms with finite group of fixed
points).   Note that $X(\FF_q)$ may not be a single orbit for $G(\FF_q)$ if the stabilizer $G_x$ is not connected.  

As noted the stabilizer scheme  $\{(g,x) | g \in G, x \in X, gx=x\}$ satisfies our hypotheses and so our results apply in this case.
In particular, if for generic $x$,  $G_x$ contains a $d$-dimensional connected unipotent subgroup, then $G_y$ does for all $y \in X$.  
Then Corollary \ref{c:finite} implies that if $y \in X(\FF_q)$ and $G_y$ has a smooth $d$-dimensional connected unipotent subgroup, then 
$G_y(\FF_q)$ contains a subgroup of order $q^d$.  
In particular, this gives lower bounds for the base size for $G(\FF_q)$ in terms of the base size of $G$ (there are examples
where the base size of the finite group can be smaller or larger although not by much -- see \cite{BGS}).

\section{Derived Subgroups} \label{s:derived} 

\subsection{\quad}
For an algebraic group $G$ defined over a field $k$,
we define $d(G)$ (resp.\ $d^0(G)$)   the dimension of the derived  group 
of the smooth  $\ol{k}$--group $G_{\ol k, red}$ (resp. the smooth connected  $\ol{k}$--group $G_{\ol k, red}^0$).
If $G$ is smooth, we have  $d(G) =\dim_k(DG)$ and  $d^0(G) = \dim_k(D(G^0))$.

It is convenient to introduce a third dimension function
$d^+(G)$  which is the supremum of the dimensions of the $C_n$'s
where $C_n$ stands for  the schematic image of 
 the commutator map $c_n: G^{2n} \to G$, $c_n(x_1, y_1, \dots, x_n, y_n)= [x_1, y_1] \dots [x_n, y_n]$.

Since the formation of the schematic image commutes with flat base change,
$d^+(G)$ is insensitive to an arbitrary field extension.
We have  $d^0(G) \leq d(G) \leq d^+(G)$  and $d(G) = d^+(G)$ for $G$ smooth.

\begin{proposition} \label{prop_plus} Let $S$ be a scheme and let $G$ be a
 flat  $S$--group scheme  of finite presentation.
Then the function $s \mapsto d^+(G_{\kappa(s)})$ is 
lower semi-continuous.
\end{proposition}

\begin{proof} Using the same kind of argument  as in the proof of Theorem \ref{thm_main}
it is enough to deal with the case $S=\Spec(A)$ where $A$ is a 
DVR with fraction field $K$ and of residue field $k$.
  Let $n \geq 1$ be an integer and  consider the commutator map $c_n: G^{2n} \to G$. Let $\gC_n  \subset G$ be the 
schematic closure of $C_{n,K}$; it is flat over $A$ so equidimensional of dimension $d$ according to 
\cite[12.1.1.5]{EGA4}.
Since $G_K^{2n}$ is dense in $G^{2n}$, 
it follows that $c_n$ factorizes through $\gC_n$ so that 
 $c_{n, k}$ factorizes through $(\gC_n)_k$. Therefore
 $C_{n,k} \subset (\gC_n)_k$ whence
 $d^+(G_k) \leq d^+(G_{K})$.
\end{proof}

\begin{corollary} \label{cor_plus} 
Let $S$ be a scheme and let $G$ be an  $S$--group scheme  of finite presentation.

\smallskip

\noindent (1) Assume that $G$ is smooth.
Then the functions $s \mapsto d(G_{\kappa(s)})$ and $s \mapsto d^0(G_{\kappa(s)})$ are
lower semi-continuous.

\smallskip

\noindent (2) Assume that $S$ is irreducible with generic point 
$\xi$ such that $G_{\kappa(\xi)}$ is smooth. Then $d(G_{\kappa(\xi)}) \geq d(G_{\kappa(s)})$ 
for each $s \in S$.

\end{corollary}

\begin{proof} 
 (1) In this case $d^+(G_{\kappa(s)})=d(G_{\kappa(s)})$ so that 
Proposition \ref{prop_plus} implies that $d$ is lower semi-continuous.
 For the other function  we consider the (smooth) $S$--group scheme $G^0$ defined in \cite[VI$_B$.3.10]{SGA3}.
 We have $d^0(G_{\kappa(s)})= d(G^0_{\kappa(s)})$ so  $d^0$ is upper semi-continuous.

\smallskip

\noindent (2) We have $d(G_{\kappa(\xi)}) = d^+(G_{\kappa(\xi)}) \geq d^+(G_{\kappa(s)}) \geq d(G_{\kappa(s)}) $.

\end{proof}

This function $d$ may  fail to be upper   semi-continuous 
even in the case of a smooth affine group scheme over a DVR with connected fibers;
in that case it would be locally constant according to Corollary \ref{cor_plus}.(1).

\begin{lemma} \label{lem_BT} Let $k$ be a an algebraically closed 
field and $G$ be a split semisimple simply connected
$k((t))$-group assumed almost simple of rank $r$. Let $\gB$ be  Bruhat-Tits  $k[[t]]$--group scheme
attached to an Iwahori subgroup of $G(k((t)))$.
Then we have $$ d(\gB_k) \leq   d(G)- r < d(G) = \dim_{k((t))}(\gB_{k((t)))}).$$
\end{lemma}

Such a  $\gB$ is smooth and has connected fibers according to \cite[prop. 4.6.32]{BT}.
 
\begin{proof} 
In this case $G_{k[[t]]}$ is a Bruhat-Tits group scheme attached to the maximal parahoric
subgroup $G(k[[t]])$ of $G(k((t)))$. The Bruhat-Tits correspondence \cite[th. 4.6.35]{BT}
is a bijection between the $k$--parabolic subgroups of $G_k$ and the parahoric
subgroups of   $G(k((t)))$ included in $G(k[[t]])$.
By taking a Borel subgroup $B_k$ of $G_k$, we get then a Iwahori subgroup
$\cB$ of $G(k((t)))$ and a Bruhat-Tits group scheme $\gB$ such that 
$B_k$ occurs as quotient of   $\gB_k$. In particular 
 $\gB_k$ maps onto a Borel $k$--subgroup  of $G_k$
so admits a commutative quotient of dimension $r$. It follows that 
$d(\gB_k) \leq   \dim(\gB_k) -r = \dim(G) - r =d(G)-r < d(G)$.
We have proven that for one specific Iwahori subgroup but 
this is enough by conjugacy reasons.
\end{proof}

We conclude by giving an example of a stabilizer scheme where the generic stabilizer is simple of dimension $3$
but some stabilizer is abelian (also of dimension $3$).  

Let $G=\Sp_4(k)=\Sp(V)$ for $k$ any algebraically closed field.  Let $X = V \oplus V$.  If $x=(v_1, v_2)$ is a generic
point, then the stabilizer of $x$ is the subgroup acting trivially on the nondegenerate $2$-space spanned by
$v_1$ and $v_2$ and so $G_x \cong \Sp_2(k)$.   If $y=(w_1, w_2)$ with $w_1$ and $w_2$ spanning a totally singular
$2$-space, then $G_y$ is the unipotent radical of the parabolic subgroup stabilizing the space spanned by $w_1$ and $w_2$.
In particular, for a generic point $x$, $G_x$  is nonsolvable while $G_y$ is abelian.

\subsection{Abelianization}
A variant is the following.
For an algebraic group $G$ defined over a field $k$,
we define $d_{ab}(G)$ (resp.\ $d^0_{ab}(G)$)   the dimension of the abelianization of 
 the smooth  $\ol{k}$--group $G_{\ol k, red}$ (resp. the smooth connected  $\ol{k}$--group $G_{\ol k, red}^0$).
We have  $d(G) + d_{ab}(G)=  \dim_k(G)$ 
and similarly  $d^0_{ab}(G) + d^0_{ab}(G)=  \dim_k(G)$.

\begin{corollary} \label{cor_ab} 
Let $S$ be a scheme and let $G$ be an  $S$--group scheme  of finite presentation.

\smallskip

\noindent (1) Assume that $G$ is smooth.
Then the functions $s \mapsto d_{ab}(G_{\kappa(s)})$ and $s \mapsto d^0_{ab}(G_{\kappa(s)})$ are
upper semi-continuous.

\smallskip

\noindent (2) Assume that $S$ is irreducible with generic point 
$\xi$ such that $G_{\kappa(\xi)}$ is smooth. Then $d_{ab}(G_{\kappa(\xi)}) \leq d_{ab}(G_{\kappa(s)})$ 
for each $s \in S$.

\end{corollary}

\begin{proof} 
 (1) Since the dimension is a locally constant function,
 the statement follows from Corollary \ref{cor_plus}.(1).    

\smallskip

\noindent (2) Corollary \ref{cor_plus}.(2)  states that $d(G_{\kappa(\xi)}) \geq d(G_{\kappa(s)})$
so that $-d(G_{\kappa(\xi)}) \leq - d(G_{\kappa(s)})$.
On the other hand we have $\dim(G_{\kappa(\xi)}) \leq \dim(G_{\kappa(s)})$ according to 
\cite[VI$_B$.4.1]{SGA3}. By summing the inequalities we obtain  $d_{ab}(G_{\kappa(\xi)}) \leq d_{ab}(G_{\kappa(s)})$ 
as desired.

\end{proof}

\section{Appendix: Alternate proof of Theorem \ref{t:unipotent rank}}\label{sec_brian}

We present an alternate proof of 
the upper semi-continuity of $d_u$.
It involves the following preliminary fact.

\begin{lemma}  \label{lem_brian} Let $F$ be a field.

\smallskip

\noindent (1) Let  $M$ be an  affine smooth connected $F$--group.
The following assertions are equivalent:

\smallskip

(i) $M$ is unipotent;

\smallskip

(i') $M_{\ol F}$ is unipotent;

\smallskip

(ii) All $F$--tori of $M$ are trivial;

\smallskip

(ii')  All $\ol{F}$--tori of $M$ are trivial;

\smallskip

\smallskip
 
 \noindent (2)  Let
$f: G \to H$ be an isogeny between affine smooth connected
$F$--groups. Then $G$ is unipotent if and only if $H$ is unipotent.
\end{lemma}

\begin{proof}
(1) Let $\rho: M \to \GL_n$ be  faithful linear representation.
 The equivalence $(i) \Longleftrightarrow
(i')$  is by definition since in both cases it 
says that $\rho( G(\ol{F}))$ consists in unipotent elements.

\smallskip

\noindent $(i') \Longrightarrow (ii')$.
Let $ \GG_{m,\ol F}^r \hookrightarrow  M_{\ol F}$ be  
a $\ol F$--subtorus. If $r  \geq 1$, we pick an element
$t \not = 1 \in (\ol F)^r$. We have that 
$\rho(t) \in GL_n( \ol{k})$ is unipotent and
semisimple so is $1$ contradicting $t \not = 1$. We conclude that  $r=1$.

\smallskip 
\noindent $(ii') \Longrightarrow (ii)$. This is obvious. 

\smallskip

\noindent $(ii) \Longrightarrow (ii')$.
According to the conjugacy theorem \cite[11.3]{B},n 
all maximal $\ol{F}$--subtori of $G_{\ol F}$ are conjugated. 
According to a result of Grothendieck \cite[18.2(i)]{B},  there is a maximal $\ol{F}$-torus $T$ in 
$M_{\ol F}$ that is defined over $F$.
Our assumption implies that $T=1$.
According to the conjugacy theorem for maximal
$\ol{F}$-subtori of $G_{\ol F}$, \cite[11.3]{B}
we conclude that all $\ol{F}$--tori of $M$ are trivial.

\smallskip

\noindent (2) By (1) we can assume $F$ is algebraically closed and it suffices to show that some maximal torus is trivial.  For a maximal torus $T$ in $G$ its image $f(T)$ is a maximal torus in $H$ by \cite[Prop. 11.14.(1)]{B}. 
Since $\ker(f)$ is finite we have $f(T)=1$ if and only if
 $T = 1$. The proof is completed.
\end{proof}

We come now to the proof of Theorem \ref{t:unipotent rank}.
The beginning is verbatim that of the proof of
Theorem \ref{thm_main}, that is, a reduction to 
the case of the spectrum of a DVR $A$
with fraction field $K$ and residue field $k$.
We are given a separated $A$--group scheme $G$ of finite presentation
and want so show the inequality 
$d_u(G_k) \geq d_u(G_K)$. We put $d=d_u(G_K)$. 
Using the insensitivity of $d_u$ by change of fields (Lemma \ref{lem_alg_closed2})
we can assume that $A$ is complete.
Also we are  allowed to make arbitrary 
finite  extensions of $K$ so that we can assume that 
$G_K$ admits a  smooth connected unipotent $K$--group 
$U$ of dimension $d$.
Letting $\gH$ be the schematic closure of $U$ in $G$,
it is separated flat over $A$ and of finite presentation.
Raynaud's affiness criterion (see \S \ref{subsec_DVR}) shows that $\gH$
is affine. Replacing $G$ by $\gH$
we can then assume that the $A$--group scheme $G$ is  affine flat, 
 and that $G_K$ is smooth unipotent connected.

We consider the case $G$ smooth.
This enables to  deal with the neutral component $G^0$ of $G$
\cite[VI$_B$.3.10]{SGA3} which is 
a smooth open $A$--subgroup $G$ such that 
$G^0_K$ (resp.\ $G^0_k$) is the neutral component of
the algebraic group $G_K$ (resp.\ $G_k$). 
Also the $A$--group scheme $G^0$ is affine according to Raynaud's criterion. 

 Lemma \ref{lem_brian}.(1) states that 
$G_k^0$ is unipotent if and only if 
all tori of $G_k^0$  are trivial.
 Let $T_0$ be a maximal  torus  of $G_k^0$. 
 Since $A$ is henselian, $T_0$ lifts to a  subtorus $T$ of $G^0$ by using Grothendieck's representability
 theorem \cite[XI.4.1]{SGA3}.
Since $G_K$ is unipotent,  $T_K$ is trivial so that $T=1$
and $T_0=1$. It follows that $G_k^0$ is unipotent so that $d_u(G)=d$.

We consider now  the general case. 
According to \cite[prop. 3.4]{PY} (based on 
\cite[app. II]{A}), there exists a 
local extension $A'$ of $A$ of DVR's
such that the normalization $\widetilde G'$
of $G'= G \times_A A'$ is smooth over $A'$
and such that the fraction field $K'$ is finite over $K$.
We denote by $k'$ the residue field of $A'$.
According to \cite[thm. A.6]{PY}
 the normalization morphism $h: \widetilde G' \to G'$
 is finite. In particular the morphism of smooth affine 
 connected $K'$--groups
 $h_{K'}^0: \bigl(\widetilde G')_{K'}\bigr)^0 \to G_{K'}$ is  an isogeny between smooth affine connected $K'$--groups.
 Since $G_{K'}$ is unipotent,  Lemma \ref{lem_brian}.(2) 
 shows that  $(\widetilde G')_{K'}^0$ is  unipotent.
 
 The smooth case applied to $(\widetilde G')^0$ over $A'$
 shows that $(\widetilde G')^0_{k'}$ is unipotent of dimension $d$.
 Since the homomorphism $h_k^0: (\widetilde G')^0_{k'} \to (G'_{k'})^0 \cong (G_k)^0 \times_k k'$ is finite, Lemma
 \ref{lem_brian}.(2) shows that the $k'$--subgroup 
 $(G'_{k'})^0/ \ker(h_{k'}^0)$ 
 of $G_{k'}$ is unipotent of dimension $d$. Thus $d_u(G_k) \geq d$.

\end{document}